\documentclass[english,12pt,oneside]{amsproc}
\usepackage[english]{babel}
\usepackage{a4wide}
\usepackage{amsthm}
\usepackage{graphics}
\usepackage{graphicx}
\usepackage{amsfonts, amssymb, amscd, amsmath}
\usepackage{latexsym}
\usepackage[matrix,arrow,curve]{xy}
\usepackage{mathabx}
\usepackage{color}
\usepackage{pbox}
\usepackage{tikz}
\usetikzlibrary{matrix,decorations.pathreplacing,positioning}
\usepackage{hyperref}
\usepackage{rotating}
\usepackage{multirow}

\DeclareMathOperator{\Cone}{Cone}

\DeclareMathOperator{\SO}{SO} \DeclareMathOperator{\Sp}{Sp}
\DeclareMathOperator{\dc}{dc} \DeclareMathOperator{\Mat}{Mat}
\DeclareMathOperator{\diag}{diag} \DeclareMathOperator{\Pe}{Pe}
\DeclareMathOperator{\conv}{conv}\DeclareMathOperator{\relint}{relint}
\DeclareMathOperator{\imm}{Im} \DeclareMathOperator{\HRe}{Re}
\DeclareMathOperator{\Dic}{Dic} \DeclareMathOperator{\coord}{coord}
\DeclareMathOperator{\Fl}{Fl} \DeclareMathOperator{\GL}{GL}

\newcommand{\jump}[1]{\ensuremath \raisebox{1pt}{$#1$}}
\newcommand{\sit}[1]{\ensuremath \raisebox{-1pt}{$#1$}}
\newcommand{\wt}[1]{\widetilde{#1}}

\newcommand{\Zo}{\mathbb{Z}}
\newcommand{\Ro}{\mathbb{R}}

\newcommand{\Co}{\mathbb{C}}

\newcommand{\Ho}{\mathbb{H}}

\newcommand{\Zt}{\Zo_2}

\newcommand{\HIm}{\overrightarrow{\imm}}

\newcommand{\htimes}{\widehat{\times}}

\newcommand{\eqd}{\stackrel{\text{\tiny def}}{=}}

\newcommand{\dd}{\partial}

\newcommand{\RP}{\mathbb{R}P}

\newcounter{stmcounter}[section]

\numberwithin{equation}{section}

\theoremstyle{plain}
\newtheorem{cor}[stmcounter]{Corollary}

\newtheorem{thm}[stmcounter]{Theorem}

\newtheorem{prop}[stmcounter]{Proposition}
\newtheorem{lem}[stmcounter]{Lemma}

\newtheorem{clai}[stmcounter]{Claim}

\theoremstyle{definition}
\newtheorem{defin}[stmcounter]{Definition}

\theoremstyle{remark}

\newtheorem{rem}[stmcounter]{Remark}
\newtheorem{con}[stmcounter]{Construction}

\begin{document}

\title{Topology of misorientation spaces}

\author{Anton Ayzenberg}
\address{Faculty of computer science, Higher School of Economics}
\email{ayzenberga@gmail.com}

\author{Dmitry Gugnin}
\address{Faculty of mechanics and mathematics, Lomonosov Moscow State University}
\email{dmitry-gugnin@yandex.ru}

\date{\today}



\begin{abstract}
Let $G_1$ and $G_2$ be discrete subgroups of $\SO(3)$. The double quotients of the form $X(G_1,G_2)=\sit{G_1}\backslash \jump{\SO(3)}/\sit{G_2}$ were introduced in material science under the name misorientation spaces. In this paper we review several known results that allow to study topology of misorientation spaces. Neglecting the orbifold structure, all misorientation spaces are closed orientable topological 3-manifolds with finite fundamental groups. In case when $G_1,G_2$ are crystallography groups, we compute the fundamental groups $\pi_1(X(G_1,G_2))$, and apply Thurston's elliptization conjecture to describe these spaces. Many misorientation spaces are homeomorphic to $S^3$ by Poincar\'{e} conjecture. The sphericity in these examples is related to the theorem of Mikhailova--Lange, which constitutes a certain real analogue of Chevalley--Shephard--Todd theorem. We explicitly describe topological types of several misorientations spaces avoiding the reference to Poincar\'{e} conjecture. Classification of misorientation spaces allows to introduce new $n$-valued group structures on $S^3$ and $\RP^3$. Finally, we outline the connection of the particular misorientation space $X(D_2,D_2)$ to integrable dynamical systems and toric topology.
\end{abstract}
\maketitle

\section{Introduction}\label{secIntro}
Let $G_1,G_2\subset \SO(3)$ be two finite subgroups in the group of special orthogonal transformations of $\Ro^3$. Consider the left action of the product group $G_1\times G_2$ on the manifold $\SO(3)$ given by $(g_1,g_2)A=g_1\cdot A\cdot g_2^{-1}$.

\begin{defin}\label{definMisorSpace}
The topological orbit space
\begin{equation}\label{eqMisorientSpaceDefin}
X(G_1,G_2)\eqd\SO(3)/(G_1\times G_2)=\sit{G_1}\backslash \jump{\SO(3)}/\sit{G_2},
\end{equation}
is called \emph{the misorientation space} of the pair $G_1$, $G_2$.
\end{defin}

Misorientation spaces appeared in material science in connection with the study of polycrystal materials. Assume that a sample of a mineral $1$ with the crystal lattice $N_1\subset \Ro^3$ contains included fractions of a mineral $2$ (probably the same) with the crystal lattice $N_2\subset \Ro^3$. The lattices of the fractions of mineral 2 may be displaced relative to the lattice of mineral 1. Roughly, a displacement is an orthogonal matrix (we neglect affine translations). However two displacements cannot be distinguished if they differ either by a (origin preserving) symmetry $g_1\in G_1$ of $N_1$, or by a symmetry $g_2\in G_2$ of $N_2$. Hence a displacement, also called a \emph{misorientation}, should be an element of the double quotient $X(G_1,G_2)=\sit{G_1}\backslash \jump{\SO(3)}/\sit{G_2}$, which motivates Definition \ref{definMisorSpace} above.

The topological structure of $X(G_1,G_2)$ is important in the applications. In a polycrystal material, there may be many fractions, misoriented differently, which gives rise to the distribution of misorientations. The distribution $\mu$ is a measure on $X(G_1,G_2)$, such that, for $A\subseteq X(G_1,G_2)$, its measure $\mu(A)$ equals the physical volume (or its percentage) of all fractions of the sample whose misorientations belong to $A$. The diagram of a misorientation distribution may provide an information about the physical properties of the sample. In order to visualize such diagram, its representation in $\Ro^3$ should be specified. Ideally, we want such diagram to admit a continuous embedding in $\Ro^3$. In this case there would exist exact and correct way to visualize the misorientations' distribution diagram. Unfortunately, such embedding never exists, as follows from Proposition~\ref{propMisorSpaceIsMfd} below.

Nevertheless, there are well known standard methods of representation for simple 3-dimensional spaces. For example, $S^3$ can be represented in $\Ro^3$ by means of stereographical projection, and $\SO(3)$ can be represented by a closed ball in $\Ro^3$ with opposite points of its boundary identified. Certainly, these ways of representation are not embeddings, but they give a way of understanding the structure of the space.

One way to solve the problem of complicated misorientation spaces was pursued by Patala and Schuh \cite{PatSch}, who introduced the notion of a \emph{homophase misorientation space}. If $G_1=G_2=G$, the homophase misorientation space is the additional quotient of $X(G,G)$ by the involution $[A]\leftrightarrow [A^{-1}]$, called \emph{grain exchange property}:
\[
\widetilde{X}(G)=X(G,G)/([A]\sim[A^{-1}]).
\]
Physically, this corresponds to the situation when two lattices are not distinguished: one cannot tell which lattice is the first and which is the second. Of course, in this case the two lattices should be indistinguishable, which explains the assumption $G_1=G_2$. The following statement was proved in~\cite{PatSch} by scrupulous analysis of the fundamental domains of the actions (proper point crystallography groups and their short notations are listed in Table~\ref{tableSubgroupsTable} below).

\begin{prop}[Patala and Schuh \cite{PatSch}]\label{propPatalaSchuh}
$\widetilde{X}(G)$ is homeomorphic to the 3-dimensional disc $D^3$ if $G$ belongs to the list: $\{D_2,D_4,D_6,T,O\}$. $\widetilde{X}(G)$ is homeomorphic to the cone over real projective plane $\Cone \RP^2$ if $G$ belongs to the list $\{C_2,C_3,C_4,C_6,D_3\}$.
\end{prop}

In the disc case, the homophase misorientation spaces admit embeddings in $\Ro^3$, the coordinates given by these embeddings are described in~\cite{PatSch}. The homophase misorientation spaces are simpler, however their meaning in material science is less transparent. To define the distribution diagram of misorientations, one needs to fix a lattice as the main reference frame: the distributions of smaller fractions are counted relatively to this frame. So far, there seems to be no analogue of misorientation distribution diagram in the homophase setting.

As was noticed in \cite{PatSch}, the misorientation spaces $X(G_1,G_2)$ are well studied in topology if one of the acting groups, say $G_2$, is trivial. In this case, the remaining action of $G_1$ by the left multiplication on $\SO(3)$ is free, so the orbit space is a connected smooth 3-manifold, which is an elliptic 3-manifold (recall that an elliptic manifold is a quotient of $S^3\cong\Sp(1)$, the double cover of $\SO(3)$, by a free action of discrete group, see~\cite{Thurs}). Certain wonderful 3-manifolds appear in this way, including the following.
\begin{enumerate}
  \item If both groups are trivial, we have $X(1,1)=\SO(3)\cong \RP^3$;
  \item If $G_1=\Zo_m$ is generated by a an $m$-rotation\footnote{$m$-rotation is an axial rotation of order $m$ in $\SO(3)$}, then $\SO(3)/G_1$ is the lens space $L(2m;1)$.
  \item If $G_1=D_2\cong \Zt^2$ is generated by axial $2$-rotations in 3 orthogonal axes, then $\SO(3)/G_1\cong \Sp(1)/Q_8\cong F_3(\Ro)$, where $Q_8=\{\pm1,\pm i,\pm j,\pm k\}$ is the quaternion group, and $F_3(\Ro)$ is the manifold of full flags in $\Ro^3$.
  \item If $G_1$ is the symmetry group of an icosahedron, then $\SO(3)/G_1$ is the famous Poincar\'{e} sphere. Certainly, $G_1$ is not a point crystallographic group, so this example does not have much meaning in material science. However, in our definition of the misorientation space, arbitrary discrete subgroups of $\SO(3)$ are allowed. The icosahedral group is interesting due to its maximal complexity, so we include this group in our exposition.
\end{enumerate}
In the following we do not consider the cases when one of the groups $G_1$ or $G_2$ is trivial.

When the two-sided action of $G_1\times G_2$ on $\SO(3)$ is not free, the space $X(G_1,G_2)$ is an elliptic orbifold, see~\cite{BLP}. However, there seems to be no precise description of the underlying topology of these orbifolds for particular choices of the groups $G_1$ and $G_2$ in the literature. Our paper aims to fill in this gap.

The topology of misorientation spaces is studied by standard methods, reviewed below. At first, we have

\begin{prop}\label{propMisorSpaceIsMfd}
For arbitrary $G_1$ and $G_2$, the misorientation space $X(G_1,G_2)$ is a closed orientable topological $3$-manifold.
\end{prop}

This follows from a more general well-known statement formulated in Proposition~\ref{propFiniteActionOnMfds} below. In Section~\ref{secMisorientationsMain}, we recall the classical Armstrong theorem, which is used to compute the fundamental groups of misorientation spaces. The fundamental groups are computed for a list of pairs $(G_1,G_2)$: this calculation is the central technical part of the paper. The groups $G_1$ and $G_2$ are taken from Table~\ref{tableSubgroupsTable}: these are the standard 10 nontrivial proper point crystallography groups and, additionally, the icosahedral group. The fundamental groups of misorientation spaces are listed in Table~\ref{tableFundGrps}, see Theorem~\ref{thmFundGrps}. We see that many misorientation spaces are simply connected, hence homeomorphic to $S^3$ according to Poincar\'{e} conjecture. It should be mentioned that elliptic orbifolds with underlying spaces homeomorphic to $S^3$ were extensively studied by Dunbar in~\cite{Dun}.

In general, the finite fundamental group of a closed 3-manifold determines its homeomorphism type uniquely,
unless the fundamental group is cyclic, in which case there may be nonhomeomorphic lens spaces. This follows from Thurston's geometrization conjecture proved by Perelman, see~\cite{Lott}. This observation allows to describe the topological types of all misorientation spaces with non-cyclic fundamental groups. Fortunately, the cyclic cases split in two categories: (1) in one category, the homeomorphism type of the corresponding lens spaces is uniquely determined by the group; (2) in the second category, the topological type can be described ``by hand''. The resulting homeomorphism types of misorientation spaces are gathered in Table~\ref{tableMisSpaces}, see Theorem~\ref{thmAllMisSpaces}. In particular, we have the following statement.

\begin{table}[h]
	\centering
\resizebox{\columnwidth}{!}{
\begin{tabular}{c||c|c|c|c|c|c|c|c|c|c|c|}
& $C_2$ & $C_3$ & $C_4$ & $C_6$ & $D_3$ & $D_2$ & $D_4$ & $D_6$ & $T$ & $O$ & $I$ \\
\hline
\hline
$C_2$ & $\RP^3$ & $L(12;5)$ & $L(4;1)$ & $L(6;1)$ & $\RP^3$ & $S^3$ & $S^3$ & $S^3$ & $L(3;1)$ & $S^3$ & $S^3$ \\
\hline
$C_3$ && $\RP^3$ & $L(24;7)$ & $L(4;1)$ & $L(4;1)$ & Free & Free & $\Fl(\Ro^3)$ & $S^3$ & $\RP^3$ & $S^3$ \\
\hline
$C_4$ &&& $\RP^3$ & $L(12;5)$ & $L(4;1)$ & $\RP^3$ & $S^3$ & $\RP^3$ & $L(6;1)$ & $S^3$ & $\RP^3$ \\
\hline
$C_6$ &&&& $\RP^3$ & $\RP^3$ & $L(3;1)$ & $L(3;1)$ & $S^3$ & $S^3$ & $S^3$ & $S^3$ \\
\hline
$D_3$ &&&&& $\RP^3$ & $S^3$ & $S^3$ & $S^3$ & $S^3$ & $S^3$ & $S^3$ \\
\hline
$D_2$ &&&&&& $S^3$ & $S^3$ & $S^3$ & $L(3;1)$ & $S^3$ & $S^3$ \\
\hline
$D_4$ &&&&&&& $S^3$ & $S^3$ & $L(3;1)$ & $S^3$ & $S^3$ \\
\hline
$D_6$ &&&&&&&& $S^3$ & $S^3$ & $S^3$ & $S^3$ \\
\hline
$T$ &&&&&&&&& $S^3$ & $S^3$ & $S^3$ \\
\hline
$O$ &&&&&&&&&& $S^3$ & $S^3$ \\
\hline
$I$ &&&&&&&&&&& $S^3$ \\
\hline
\end{tabular}
}
 \caption{Topology of misorientation spaces. The positions $X(C_3,D_2)$ and $X(C_3,D_4)$, filled with ``Free'', are explained in Remark~\ref{remFreeExplain}.}\label{tableMisSpaces}
\end{table}

\begin{prop}\label{propSameGroup}
The misorientation space $X(G,G)$ is homeomorphic to a 3-sphere, if $G$ belongs to the list: $\{D_2,D_4,D_6,T,O,I\}$. The misorientation space $X(G,G)$ is homeomorphic to $\RP^3$, if $G$ belongs to the list: $\{C_2,C_3,C_4,C_6,D_3\}$.
\end{prop}

This result is consistent with Proposition~\ref{propPatalaSchuh}. The involution $[A]\mapsto [A^{-1}]$ on $X(G,G)$ is non-free: all 2-rotations represent the fixed points stratum of dimension~2. It is natural to expect that the quotient $\widetilde{X}(G)=X(G,G)/([A]\sim[A^{-1}])$ is a disk if $X(G,G)$ is a 3-sphere, and the space $\widetilde{X}(G)$ is $\Cone \RP^2$ if $X(G,G)$ is homeomorphic to $\RP^3$. This argument is explained in detail in Remark~\ref{remInvolution}.

Proposition~\ref{propSameGroup} allows to introduce nontrivial $n$-valued group structures on the manifolds $S^3$ and $\RP^3$. We review the related constructions in Subsection~\ref{subsecNval}.

Certain particular consequences can be derived from Table~\ref{tableMisSpaces}. For example, the homeomorphism $X(C_2,I)\cong S^3$ implies that there exists an involution on the Poincar\'{e} sphere whose orbit space is homeomorphic to $S^3$. This fact is not surprising: it is known that Poincar\'{e} sphere is a branched 2-fold covering over $S^3$ ramified over a toric knot $T(3,5)$, see \cite{KirSch}.
%

The contents of Table~\ref{tableMisSpaces} rely on the elliptization conjecture. The particular cases when $X(G_1,G_2)\cong S^3$ follow from Poincar\'{e} conjecture. This looks like shooting sparrows with a cannon: the proof of Poincar\'{e} conjecture does not provide explicit spherical coordinates, and cannot be used further in the study of misorientation distribution diagrams. To find convenient coordinates (in the cases when $X(G_1,G_2)$ is homeomorphic to $S^3$ or $\RP^3$), is a separate problem that we want to address.

In Section~\ref{secCoordinatesSpaces}, we construct explicit coordinates and describe the homeomorphism type avoiding the reference to Poincar\'{e} conjecture for several misorientation spaces. Namely, we explicitly prove the homeomorphisms $X(C_n,C_n)\cong\RP^3$, $X(D_n,D_n)\cong \RP^3$ for odd $n$, and $X(D_n,D_n)\cong S^3$ for even $n$, see Propositions~\ref{propCyclicCoords}, \ref{propCyclicCoordsDcaseOdd}, and \ref{propCyclicCoordsDcaseEven}.

There is a straightforward relation of spherical misorientation spaces to invariant theory of discrete transformation groups. More precisely, the pair of groups $G_1,G_2\subset \SO(3)$ gives rise to the discrete subgroup
\begin{equation}\label{eqDiscreteSO4}
2G_1\htimes 2G_2\subset(\Sp(1)\times \Sp(1))/\{\pm1\}\cong \SO(4).
\end{equation}
of orthogonal transformations of $\Ro^4$. Here $2G_1$ and $2G_2$ are the binary extensions of $G_1$ and $G_2$ in the group $\Sp(1)\cong S^3$ of unit quaternions; $2G_1\htimes 2G_2=(2G_1\times 2G_2)/\{\pm1\}$; and $(\Sp(1)\times \Sp(1))/\{\pm1\}\cong \SO(4)$ is the standard isomorphism of Lie groups given by the two-sided multiplication action of $\Sp(1)\times\Sp(1)$ on $\Ho\cong\Ro^4$.

We show that $\pi_1(X(G_1,G_2))$ is trivial if and only if the representation of $2G_1\htimes 2G_2$ on $V\cong \Ro^4$ is generated by rotations, see Corollary~\ref{corSpherePseudoref}. A rotation is an orientation preserving orthogonal transformation which is identical on a real codimension 2 subspace. The following theorem first appeared in the work of Mikhailova~\cite{Mikh}, however its complete proof was given by Lange~\cite{Lange16} based on the complete classification of finite rotation groups obtained by Lange and Mikhailova~\cite{LangMikh}.

\begin{prop}[Lange--Mikhailova]\label{propMikh}
If $V\cong \Ro^n$, and a finite subgroup $G\subset \GL(V)$ is generated by rotations, then $V/G$ is homeomorphic to $V$.
\end{prop}

The converse statement holds true if $\dim V=4$, see Proposition~\ref{propCriter1con}. 
Note that the homeomorphism $S^{n-1}/G\cong S^{n-1}$ implies the homeomorphism $\Ro^n/G\cong \Ro^n$. Hence, in the case $n=4$, Proposition~\ref{propMikh} follows from Poincar\'{e} conjecture.
The classification of all finite rotation groups obtained in~\cite{LangMikh}, allowed Lange~\cite{Lange16} to prove that $V/G$, taken with a natural triangulation, is a PL-manifold without boundary if and only if $G$ is generated by rotations, and in this case $V/G\cong V$ and $S^{n-1}/G\cong S^{n-1}$.

Unfortunately, the homeomorphism $V/G\cong V$ of Proposition~\ref{propMikh} is not always the consequence of the fact that the algebra of invariants $\Ro[V]^G$ is the polynomial algebra in $n=\dim V$ generators, see subsection~\ref{subsecInvar} (the corresponding statement holds in the complex case according to Chevalley--Shephard--Todd theorem~\cite{ShepTodd}). However, in the cases when $\Ro[V]^G$ is a polynomial algebra $\Ro[f_1,\ldots,f_n]$, the homogeneous polynomials $f_1,\ldots,f_n$ can be used to construct the spherical coordinates on $S^{n-1}/G\cong S^{n-1}$ given by the real analytic functions, see Construction~\ref{conCoordInvar}.

The particular misorientation space $X(D_2,D_2)$ appeared independently in many areas of mathematics. In Section~\ref{secD2detailed}, we review several ways to understand the topology of this space. The first proof of the homeomorphism $X(D_2,D_2)\cong S^3$ is contained implicitly in the work of van Moerbeke \cite{VanM} on the dynamical properties of the periodic Toda lattice. Another way to understand this homeomorphism is related to toric topology and symplectic geometry: its complex version: $\sit{T^3}\backslash \jump{U(3)}/\sit{T^3}\cong S^4$ was proved by Buchstaber--Terzic~\cite{BT2} and, independently, by the first author~\cite{AyzMatr}. Finally, the extensive study of $X(D_2,D_2)$ from the perspective of invariant theory was done by Mikhailova in~\cite{Mikh}. The case of $X(D_2,D_2)$ is the most difficult, and in some sense, the most essential part of the study of discrete transformations of $\Ro^4$.
%

%

\section{Misorientation spaces}\label{secMisorientationsMain}

\subsection{Finite group actions on 3-manifolds}\label{subsecGeneralActions}

In this subsection we collect several known results about non-free finite group actions on 3-manifolds in general. The following theorem is well known (see, e.g. \cite[Thm.2.5]{CHK}).

\begin{prop}\label{propFiniteActionOnMfds}
Let $M^3$ be a smooth closed orientable 3-manifold, and a finite group $G$ acts on $M$ by orientation preserving diffeomorphisms. Then the quotient space $M^3/G$ is a topological 3-manifold.
\end{prop}

The proof also follows from the result of Mikhailova, Proposition~\ref{propMikh}. In order to show that $M^3/G$ is a manifold, one needs to prove the local statement: that is the homeomorphism $\Ro^3/G\cong \Ro^3$ for any discrete subgroup $G\subset\SO(3)$. However, every element of $\Ro^3$ is a rotation, so the homeomorphism $\Ro^3/G\cong \Ro^3$ follows from Proposition~\ref{propMikh}. Note that in dimension 3, every discrete subgroup of $\SO(3)$ is the index two subgroup of some group generated by reflections, therefore, in this case Proposition~\ref{propFiniteActionOnMfds} is the instance of the following statement.

\begin{prop}[{Mikhailova, \cite[Thm.1.1]{Mikh}}]\label{propMikhIndex2}
Let $\Gamma$ be a finite subgroup generated by reflections in the space $V\cong \Ro^n$, and let $G=\Gamma^+$ be the subgroup of $\Gamma$ generated by pairwise products of reflections of $\Gamma$ (equiv. $G$ is the subgroup of orientation preserving transformations from $\Gamma$). Then $G$ is generated by rotations, and we have the homeomorphisms $V/G\cong V$ and $S^{n-1}/G\cong S^{n-1}$.
\end{prop}

This statement is easier than Proposition~\ref{propMikh}: it can be proved as follows. The action of $\Gamma$ on $S^{n-1}$ has fundamental domain $P_+$, some spherical polytope. The quotient $S^{n-1}/G$ can be obtained by attaching two identical copies of $P_+$ along their boundary, which certainly results in a sphere $S^{n-1}$. A somewhat similar statement appeared in the work~\cite{Gug} of the second author, where he considered the finite group action on the product of spheres.

Proposition~\ref{propFiniteActionOnMfds} implies Proposition~\ref{propMisorSpaceIsMfd}, since the multiplication action on $\SO(3)$ from either left or right side preserves its orientation.

Another useful and famous result is Armstrong theorem~\cite{Arm}. We state the theorem for finite groups, which is weaker than the original formulation given by Armstrong.

\begin{prop}[{Armstrong, \cite[Thm.3]{Arm}}]\label{propArmstrong}
Let $X$ be a connected simply connected polyhedron, and suppose a finite group $G$ acts on $X$ by simplicial transformations. Let $H$ be the subgroup of $G$, generated by all $h\in G$ such that the fixed point set $X^h$ is nonempty. Then $\pi_1(X/G)$ is isomorphic to the quotient $G/H$.
\end{prop}

\begin{cor}\label{corPi1}
Assume that a finite group $G$ acts on $S^3$ preserving the orientation. Then $\pi_1(S^3/G)\cong G/H$, where $H$ is the subgroup of $G$ generated by the union of all stabilizers $G_x$, $x\in S^3$, of the $G$-action.
\end{cor}

\begin{rem}\label{remNormalAutomatically}
Note that the subgroup $H$ generated by all stabilizers is normal. Indeed, if $g\in G_x$, then $hgh^{-1}\in G_{hx}$.
\end{rem}

%
%

%

\subsection{Misorientation spaces and double covers}

Recall that, for a pair of discrete subgroups $G_1,G_2\subset \SO(3)$, the misorientation space $X(G_1,G_2)=\sit{G_1}\backslash \jump{\SO(3)}/\sit{G_2}$ is defined. Everything can be lifted to the 3-sphere $S^3\cong\Sp(1)$ which forms the universal double cover over $\SO(3)$. We refer to~\cite{DuVal} for the exposition of the relation between quaternions, orthogonal transformations and finite groups. Let $2G_1$ and $2G_2$ be the preimages of the subgroups $G_1$ and $G_2$ of $\SO(3)$ respectively under the double cover $\dc\colon S^3\to\SO(3)$. We have
\[
X(G_1,G_2)=\sit{2G_1}\backslash \jump{S^3}/\sit{2G_2}=S^3/(2G_1\times 2G_2),
\]
where the left action of $2G_1\times 2G_2$ on $S^3$ is given by
\begin{equation}\label{eqTwoSided}
(g_1,g_2)x=g_1xg_2^{-1}.
\end{equation}

Conversely, given two discrete subgroups $\widetilde{G}_1,\widetilde{G}_2\subset S^3$ we can consider the two-sided quotient $\sit{\widetilde{G}_1}\backslash \jump{S^3}/\sit{\widetilde{G}_2}$. It coincides with some misorientation space if at least one of the subgroups $\widetilde{G}_1$ or $\widetilde{G}_2$ contains $-1\in S^3$.

\begin{con}\label{conSO4}
Generally, there is a representation of $\Sp(1)\times\Sp(1)$ on $\Ho\cong\Ro^4$ given by~\eqref{eqTwoSided}. The element $(-1,-1)\in \Sp(1)\times \Sp(1)$ generates the noneffective kernel of this representation. Every orthogonal transformation of $\Ro^4$ is written as~\eqref{eqTwoSided} for some $g_1,g_2\in\Sp(1)$, see e.g.~\cite[Sect.17]{DuVal}. These facts prove the well-known isomorphism of the Lie groups: $\SO(4)\cong (\Sp(1)\times\Sp(1))/\langle(-1,-1)\rangle$.

In general, if $\widetilde{G}_1,\widetilde{G}_2$ are subgroups of $\Sp(1)$, we consider the subgroup $\widetilde{G}_1\htimes \widetilde{G}_2$ of $\SO(4)$ defined by
\begin{equation}\label{eqHtimes}
\widetilde{G}_1\htimes \widetilde{G}_2\eqd (\widetilde{G}_1\times \widetilde{G}_2)/\langle(-1,-1)\rangle\subseteq (\Sp(1)\times \Sp(1))/\langle(-1,-1)\rangle\cong\SO(4).
\end{equation}
With this notation in mind, the misorientation space $X(G_1,G_2)$ coincides with the quotient space of $S^3\subset\Ro^4$ by the (possibly non-free) action of the discrete subgroup $2G_1\htimes2G_2\subset\SO(4)$.
\end{con}

\begin{rem}
Proposition~\ref{propFiniteActionOnMfds} and Corollary \ref{corPi1} imply that any misorientation space is a closed 3-manifold with finite fundamental group. Due to elliptization conjecture~\cite{Lott}, this means that any misorientation space is homeomorphic to a spherical 3-manifold.

By definition, a spherical 3-manifold is a quotient of the round sphere $S^3\subset \Ro^4$ by a finite subgroup of $\SO(4)$ acting freely on $S^3$. On the other hand, $\SO(4)\cong (\Sp(1)\times\Sp(1))/\{\pm1\}$, so far any finite group action on $S^3$ can be modeled by two-sided action in quaternionic form $x\mapsto g_1xg_2^{-1}$ for some discrete subgroup $G\subset \Sp(1)\times \Sp(1)$ containing $(-1,-1)$. The classification of spherical 3-manifolds \cite[Thm.4.4.14]{Thurs} is based on the classification of finite subgroups $G$ in $\Sp(1)\times \Sp(1)$ such that $G/\{\pm1\}$ acts freely on $S^3$. For a subgroup $G$ of the form $G_1\times G_2$ acting freely on $S^3$, we can see that, conversely, the spherical 3-manifold is the misorientation space $X(G_1,G_2)$.
\end{rem}

\begin{rem}
For convenience, the most important discrete subgroups of $\SO(3)$ together with their common notations in crystallography and the corresponding crystal classes are listed in Table~\ref{tableSubgroupsTable}. We also write their binary extensions in the sphere $\Sp(1)\cong S^3$, by listing the elements of these extensions. The list of groups coincides with the one studied in \cite{PatSch} (i.e. 11 proper point crystallography groups), however we added one more group, the icosahedral group. This group may not have meaning in crystallography, however it is interesting on its own.
\end{rem}

\begin{table}[p]
	\centering
	\rotatebox{90}{
		\begin{minipage}{\textheight}
\begin{tabular}{p{1cm}||p{2cm}|p{4cm}|p{3cm}|p{2cm}|p{6cm}|}
$G$& Common notation & Generated by& Corresponds to crystal classes (Sch\"{o}nflies notation)
& Binary extension $2G$ & Elements of $2G\subset\Sp(1)$\\
\hline
\hline
$C_1$ & $C_1$ & trivial group & $C_1$, $S_2$, $C_{1h}$ & $C_2\cong \Zo_2$ & $\pm1$\\
\hline
$C_2$ & $C_2(2)$ & 2-rotation & $C_2$, $C_{2h}$, $C_{2v}$, $S_4$ & $C_4\cong \Zo_4$ & $\pm1$, $\pm i$\\
\hline
$C_3$ & $C_3(3)$ & 3-rotation & $C_3$, $C_{3i}$, $C_{3v}$, $C_{3h}$ & $C_6\cong \Zo_6$ & $\pm1$, $\pm\frac{1}{2}\pm\frac{\sqrt{3}}{2}i$\\
\hline
$C_4$ & $C_4(4)$ & 4-rotation & $C_4$, $C_{4h}$, $C_{4v}$ & $C_8\cong \Zo_8$ & $\pm1$, $\pm i$, $\pm\frac{\sqrt{2}}{2}\pm\frac{\sqrt{2}}{2}i$\\
\hline
$C_6$ & $C_6(6)$ & 6-rotation & $C_6$, $C_{6h}$, $C_{6v}$ & $C_{12}\cong \Zo_{12}$ & $\pm1$, $\pm i$, $\pm\frac{1}{2}\pm\frac{\sqrt{3}}{2}i$, $\pm\frac{\sqrt{3}}{2}\pm\frac{1}{2}i$\\
\hline
$D_3$ & $D_3(32)$ & 3,2-rotations in two orthogonal axes & $D_3$, $D_{3d}$, $D_{3h}$ & $\Dic_3$ & $\pm1$, $\pm\frac{1}{2}\pm\frac{\sqrt{3}}{2}i$, $\pm j$, $\pm\frac{1}{2}j\pm\frac{\sqrt{3}}{2}k$\\
\hline
$D_2$ & $D_2(222)$ & 2-rotations in three orthogonal axes & $D_2$, $D_{2h}$, $D_{2d}$ & $2D_2=Q_8$ & $\pm1$, $\pm i$, $\pm j$, $\pm k$\\
\hline
$D_4$ & $D_4(422)$ & 4,2,2-rotations in three orthogonal axes & $D_4$, $D_{4h}$ & $2D_4$ & $\pm1$, $\pm i$, $\pm j$, $\pm k$, $\pm\frac{\sqrt{2}}{2}\pm\frac{\sqrt{2}}{2}i$, $\pm\frac{\sqrt{2}}{2}j\pm\frac{\sqrt{2}}{2}k$\\
\hline
$D_6$ & $D_6(622)$ & 6,2,2-rotations in three orthogonal axes & $D_6$, $D_{6h}$ & $2D_6$ & $\pm1$, $\pm i$, $\pm j$, $\pm k$,
$\pm\frac{1}{2}\pm\frac{\sqrt{3}}{2}i$, $\pm\frac{1}{2}j\pm\frac{\sqrt{3}}{2}k$, $\pm\frac{\sqrt{3}}{2}\pm\frac{1}{2}i$, $\pm\frac{\sqrt{3}}{2}j\pm\frac{1}{2}k$\\
\hline
$T$ & $T(23)$ & Symmetries of a regular tetrahedron & $T$, $T_h$, $T_d$ & $2T$ & $\pm1$, $\pm i$, $\pm j$, $\pm k$, $\frac{1}{2}(\pm1\pm i\pm j\pm k)$\\
\hline
$O$ & $O(432)$ & Symmetries of a regular octahedron & $O$, $O_h$ & $2O$ & $\pm1$, $\pm i$, $\pm j$, $\pm k$, $\frac{1}{2}(\pm1\pm i\pm j\pm k)$, and all permutations of coordinates in  $\pm\frac{\sqrt{2}}{2}\pm\frac{\sqrt{2}}{2}i$\\
\hline
$I$ & $I(532)$ & Symmetries of a regular icosahedron & none & $2I$ & $\pm1$, $\pm i$, $\pm j$, $\pm k$, $\frac{1}{2}(\pm1\pm i\pm j\pm k)$, and even permutations of coordinates in $(\pm i\pm \varphi j\pm\varphi^{-1}k)/2$, where $\varphi=\frac{1+\sqrt{5}}{2}$.\\
\hline
\end{tabular}
 \caption{Subgroups of $\SO(3)$, relevant to our study, and their binary extensions}\label{tableSubgroupsTable}
\end{minipage}
}
\end{table}

\begin{table}[h]
	\centering
\begin{tabular}{c||c|c|c|c|c|c|c|c|c|c|c|}
$\pi_1$& $C_2$ & $C_3$ & $C_4$ & $C_6$ & $D_3$ & $D_2$ & $D_4$ & $D_6$ & $T$ & $O$ & $I$ \\
\hline
\hline
$C_2$ & $\Zo_2$ & $\Zo_{12}$ & $\Zo_4$ & $\Zo_6$ & $\Zo_2$ & $1$ & $1$ & $1$ & $\Zo_3$ & $1$ & $1$ \\
\hline
$C_3$ && $\Zo_2$ & $\Zo_{24}$ & $\Zo_4$ & $\Zo_4$ & $\Zo_3\times Q_8$ & $\Zo_3\times2D_4$ & $Q_8$ & $1$ & $\Zo_2$ & $1$ \\
\hline
$C_4$ &&& $\Zo_2$ & $\Zo_{12}$ & $\Zo_4$ & $\Zo_2$ & $1$ & $\Zo_2$ & $\Zo_6$ & $1$ & $\Zo_2$ \\
\hline
$C_6$ &&&& $\Zo_2$ & $\Zo_2$ & $\Zo_3$ & $\Zo_3$ & $1$ & $1$ & $1$ & $1$ \\
\hline
$D_3$ &&&&& $\Zo_2$ & $1$ & $1$ & $1$ & $1$ & $1$ & $1$ \\
\hline
$D_2$ &&&&&& $1$ & $1$ & $1$ & $\Zo_3$ & $1$ & $1$ \\
\hline
$D_4$ &&&&&&& $1$ & $1$ & $\Zo_3$ & $1$ & $1$ \\
\hline
$D_6$ &&&&&&&& $1$ & $1$ & $1$ & $1$ \\
\hline
$T$ &&&&&&&&& $1$ & $1$ & $1$ \\
\hline
$O$ &&&&&&&&&& $1$ & $1$ \\
\hline
$I$ &&&&&&&&&&& $1$ \\
\hline
\end{tabular}
 \caption{The fundamental groups of misorientation spaces: $\pi_1(X(G_1,G_2))\cong (2G_1\times 2G_2)/H$}\label{tableFundGrps}
\end{table}

Now we compute the fundamental groups of misorientation spaces using Armstrong theorem.

\begin{thm}\label{thmFundGrps}
The fundamental groups $\pi_1(X(G_1,G_2))$, for different choices of $G_1,G_2\subset \SO(3)$, are as listed in Table \ref{tableFundGrps}.
\end{thm}

\begin{rem}
There holds $X(G_1,G_2)\cong X(G_2,G_1)$, with the homeomorphism given by $[A]\mapsto [A^{-1}]$. Therefore Tables~\ref{tableFundGrps} and~\ref{tableMisSpaces} can be filled symmetrically below the diagonal.
\end{rem}

For each group $G_i$ we write the corresponding elements of its binary extension $2G_i\subset \Sp(1)$ as the set of quaternions of unit length (see the last column of Table \ref{tableSubgroupsTable}). Then we use Corollary~\ref{corPi1} to compute $\pi_1(X(G_1,G_2))=\pi_1(S^3/(2G_1\times 2G_2))$. The element $(g_1,g_2)\in 2G_1\times 2G_2$ stabilizes a quaternion $x\in S^3$ if and only if $g_1x=xg_2$. In other words, $g_1,g_2$ should lie in the same conjugacy class in $\Sp(1)\subset\Ho$. The following lemma is simple but quite useful in the computations.

\begin{lem}\label{lemConjugacyCriterion}
Two quaternions $g_1,g_2$ of unit length lie in the same conjugacy class in $S^3\subset\Ho$ if and only if $\HRe g_1=\HRe g_2$.
\end{lem}

\begin{proof}
Assume $g_2=x^{-1}g_1x$ for some $x\in \Sp(1)$. Writing each $g_i$ as $\HRe g_i+\HIm g_i$, we see
\[
\HRe g_2+\HIm g_2 = x^{-1}(\HRe g_1+\HIm g_1)x=x^{-1}x\HRe g_1+ x^{-1}\HIm g_1 x = \HRe g_1+ x^{-1}\HIm g_1 x
\]
since the conjugation preserves the set of purely imaginary quaternions. This proves the only if part. On the contrary, if $\HRe g_1=\HRe g_2$, then $\|g_1\|=\|g_2\|$ implies $\|\HIm g_1\|=\|\HIm g_2\|$. The conjugation action acts by orthogonal transformations on purely imaginary quaternions. There exist orthogonal transformations which takes a vector $\HIm g_1\in \Ro^3$ to any other vector $\HIm g_2\in\Ro^3$ of the same length. This implies the if part.
\end{proof}

We give a recipe how to find the subgroup $H$, normally generated by all pairs $(g_1,g_2)\in 2G_1\times 2G_2$, for which $g_1$ and $g_2$ lie in the same conjugacy class in $\Sp(1)$. This will give the answer $\pi_1(X(G_1,G_2))\cong (2G_1\times 2G_2)/H$. The strategy is the following: we consider the subgroups $\wt{H}_1=\{g\in 2G_1\mid (g,e)\in H\}$ and $\wt{H}_2=\{g\in 2G_2\mid (e,g)\in H\}$ and construct certain subgroups $H_1\subseteq \wt{H}_1$ and $H_2\subseteq \wt{H}_2$, see Construction~\ref{conH12defin} below. Then, obviously, $H\supseteq \wt{H}_1\times \wt{H}_2$. For many interesting cases, it can be directly proved that $H_1=G_1$ and $H_2=G_2$, which already implies $\pi_1(X(G_1,G_2))=1$. In the remaining cases we describe the quotient groups $2G_1/H_1$ and $2G_2/H_2$ and their generators. We have an epimorphism
\[
(2G_1/H_1)\times(2G_2/H_2)\twoheadrightarrow (2G_1\times 2G_2)/H\cong \pi_1(X(G_1,G_2)).
\]
The kernel of this epimorphism is then described by analysing which pairs of elements' representatives of $(2G_1/H_1)\times(2G_2/H_2)$ lie in $H$.

\begin{lem}\label{lemProduct}
Assume that $g_1\in 2G_1$ is represented as the product $g_1'g_1''$, where $\HRe g_1'=\HRe g_1''$ and there exists $g_2\in G_2$ such that $\HRe g_2=\HRe g_1'$. Then $g_1\in \wt{H}_1$.
\end{lem}

\begin{proof}
We have $(g_1',g_2)\in H$ by Lemma \ref{lemConjugacyCriterion}. Now note that $\HRe g_2^{-1}=\HRe \overline{g_2}=\HRe g_2$. Again, by Lemma \ref{lemConjugacyCriterion} we have $(g_1'',g_2^{-1})\in H$. Multiplying $(g_1',g_2)$ by $(g_1'',g_2^{-1})$ we get $(g_1'g_1'',e)\in H$, which means $g_1=g_1'g_1''\in \wt{H}_1$.
\end{proof}

\begin{con}\label{conH12defin}
Let $H_1$ be the subgroup of $2G_1$ normally generated by all possible products $g_1'g_1''$ for which there exists $g_2\in G_2$ such that $\HRe g_1'=\HRe g_1''=\HRe g_2$. According to Lemma~\ref{lemProduct} we have $H_1\subseteq\wt{H}_1$. The subgroup $H_2\subseteq \wt{H}_2\subseteq G_2$ is defined in a similar way.
\end{con}

For each group $G_1$ taken from Table~\ref{tableSubgroupsTable}, we describe the possibilities for the group $H_1$ depending on the structure of $G_2$, see Table~\ref{tableSubgroupsH}. The content of Table~\ref{tableSubgroupsH} is a result of direct computation in all cases: we check all possible real parts of elements of the binary extension $2G_1$, and consider the group $H_1$ normally generated by pairwise products of elements with the given real part.

\begin{rem}
We make a few remarks. First, all binary extensions of subgroups $G_2\subset \SO(3)$ contain the element $-1$. Hence if there exists $g_2\in 2G_2$ with $\HRe g_2=a$ then there exists $g_2\in 2G_2$ with $\HRe g_2=-a$, so the sign in the conditions of Table~\ref{tableSubgroupsH} is irrelevant. Second, it is a straightforward check that an element $g\in\Sp(1)$ acts on $\Ro^3$ by axial rotation with angle $\varphi$ such that $\cos\varphi=\HRe g$. Therefore the condition $\exists g_2\in 2G_2\colon \HRe g_2=0$ is equivalent to saying that $G_2$ contains a 2-rotation. Similarly, we have
\begin{gather*}
\exists g_2\in 2G_2\colon \HRe g_2=\frac{1}{2} \Leftrightarrow G_2 \mbox{ contains a 3-rotation}. \\
\exists g_2\in 2G_2\colon \HRe g_2=\frac{\sqrt{2}}{2} \Leftrightarrow G_2 \mbox{ contains a 4-rotation}. \\
\exists g_2\in 2G_2\colon \HRe g_2=\frac{\sqrt{3}}{2} \Leftrightarrow G_2 \mbox{ contains a 6-rotation}. \\
\exists g_2\in 2G_2\colon \HRe g_2=\frac{\varphi}{2} \mbox{ or } \frac{\varphi^{-1}}{2} \Leftrightarrow G_2 \mbox{ contains a 5-rotation}.
\end{gather*}
In the latter case, $\varphi$ is the golden ratio. This case, however, does not appear in Table~\ref{tableSubgroupsH} among the conditions: the real part $\frac{\varphi}{2}$ can only appear in the case $G_2=I$, but in this case the group contains a rotation of order $2$, and this already gives the answer.
\end{rem}

\begin{table}[h!]
	\centering
	\begin{tabular}{c||p{8cm}|p{4cm}|}
$G_1$ 
& The subgroup $H_1$ depending on $G_2$ & $2G_1/H_1$ and its generators\\
\hline
\hline
\multirow{2}{*}{$C_2$ 
}
    & $\{\pm 1\}$ if $\exists g_2\in 2G_2 \colon \HRe g_2=0$ & $\Zt$, $[i]$ \\ \cline{2-3}
    & $\{+1\}$ otherwise & $\Zo_4$, $[i]$ \\
\hline
\hline
\multirow{2}{*}{$C_3$ 
}
    & $\{1,-\frac{1}{2}\pm\frac{\sqrt{3}}{2}i\}$ if $\exists g_2\in 2G_2 \colon \HRe g_2=\pm\frac{1}{2}$ & $\Zt$, $[\frac{1}{2}+\frac{\sqrt{3}}{2}i]$ \\ \cline{2-3}
    & $\{+1\}$ otherwise & $\Zo_6$, $[\frac{1}{2}+\frac{\sqrt{3}}{2}i]$ \\
\hline
\hline
\multirow{3}{*}{$C_4$ 
}
    & $\{\pm1,\pm i\}$ if $\exists g_2\in 2G_2 \colon \HRe g_2=\pm\frac{\sqrt{2}}{2}$ & $\Zt$, $[\frac{\sqrt{2}}{2}+\frac{\sqrt{2}}{2}i]$ \\ \cline{2-3}
    & $\{\pm1\}$ else if $\exists g_2\in 2G_2 \colon \HRe g_2=0$ & $\Zo_4$, $[\frac{\sqrt{2}}{2}+\frac{\sqrt{3}}{2}i]$ \\
    \cline{2-3}
    & $\{+1\}$ otherwise & $\Zo_8$, $[\frac{\sqrt{2}}{2}+\frac{\sqrt{3}}{2}i]$ \\
\hline
\hline
\multirow{4}{*}{$C_6$ 
}
    & $\{\pm1,\pm \frac{1}{2}\pm\frac{\sqrt{3}}{2}i\}$ if $\exists g_2\in 2G_2 \colon \HRe g_2=\pm\frac{\sqrt{3}}{2}$ & $\Zt$, $[\frac{\sqrt{3}}{2}+\frac{1}{2}i]$ \\
    \cline{2-3}
    & $\{1,-\frac{1}{2}\pm\frac{\sqrt{3}}{2}i\}$ else if $\exists g_2\in 2G_2 \colon \HRe g_2=\pm\frac{1}{2}$ & $\Zo_4$, $[\frac{\sqrt{3}}{2}+\frac{1}{2}i]$ \\
    \cline{2-3}
    & $\{\pm 1\}$ else if $\exists g_2\in 2G_2 \colon \HRe g_2=0$ & $\Zo_6$, $[\frac{\sqrt{3}}{2}+\frac{1}{2}i]$ \\
    \cline{2-3}
    & $\{+1\}$ otherwise & $\Zo_{12}$, $[\frac{\sqrt{3}}{2}+\frac{1}{2}i]$ \\
\hline
\hline
\multirow{3}{*}{$D_3$ 
}
    & $\{\pm1,\pm \frac{1}{2}\pm\frac{\sqrt{3}}{2}i\}$ if $\exists g_2\in 2G_2 \colon \HRe g_2=0$ & $\Zt$, $[j]$ \\
    \cline{2-3}
    & $\{+1,-\frac{1}{2}\pm\frac{\sqrt{3}}{2}i\}$ else if $\exists g_2\in 2G_2 \colon \HRe g_2=\pm\frac{1}{2}$ & $\Zo_4$, $[j]$ \\ \cline{2-3}
    & $\{+1\}$ otherwise & $2D_3$ \\
\hline
\hline
\multirow{2}{*}{$D_2$ 
}
    & $2D_2=Q_8$ if $\exists g_2\in 2G_2 \colon \HRe g_2=0$ & $1$ \\
    \cline{2-3}
    & $\{+1\}$ otherwise & $2D_2=Q_8$ \\
\hline
\hline
\multirow{2}{*}{$D_4$ 
}
    & $2D_4$ if $\exists g_2\in 2G_2 \colon \HRe g_2=0$ & $1$ \\
    \cline{2-3}
    & $\{+1\}$ otherwise & $2D_4$ \\
\hline
\hline
\multirow{4}{*}{$D_6$ 
}
    & $2D_6$ if $\exists g_2\in 2G_2 \colon \HRe g_2=0$ & $1$ \\
    \cline{2-3}
    & $2D_3$ else if $\exists g_2\in 2G_2 \colon \HRe g_2=\pm\frac{\sqrt{3}}{2}$ & $\Zo_2\times\Zo_2$, $[i],[j]$ \\ \cline{2-3}
    & $\{+1,-\frac{1}{2}\pm\frac{\sqrt{3}}{2}i\}$ else if $\exists g_2\in 2G_2 \colon \HRe g_2=\pm\frac{1}{2}$ & $Q_8$, $[i],[j]$ \\ \cline{2-3}
    & $\{+1\}$ otherwise & $2D_6$ \\
\hline
\hline
\multirow{3}{*}{$T$ 
}
    & $2T$ if $\exists g_2\in 2G_2 \colon \HRe g_2=\pm\frac{1}{2}$ & $1$ \\
    \cline{2-3}
    & $Q_8=2D_2$ else if $\exists g_2\in 2G_2 \colon \HRe g_2=0$ & $\Zo_3$, $[\frac{1}{2}(1+i+j+k)]$ \\ \cline{2-3}
    & $\{+1\}$ otherwise & $2T$ \\
\hline
\hline
\multirow{3}{*}{$O$ 
}
    & $2O$ if $\exists g_2\in 2G_2 \colon \HRe g_2=0$ & $1$ \\
    \cline{2-3}
    & $2T$ else if $\exists g_2\in 2G_2 \colon \HRe g_2=\pm\frac{1}{2}$ & $\Zo_2$, $[\frac{\sqrt{2}}{2}+\frac{\sqrt{2}}{2}i]$ \\
    \cline{2-3}
    & $\{+1\}$ otherwise & $2O$ \\
\hline
\hline
\multirow{2}{*}{$I$ 
}
    & $2I$ if $\exists g_2\in 2G_2 \colon \HRe g_2=0$ or $\pm\frac{1}{2}$ & $1$ \\
    \cline{2-3}
    & $\{+1\}$ otherwise & $2I$ \\
\hline
\hline
\end{tabular}
 \caption{Subgroups $H_1\subseteq 2G_1$ for different choices of $G_2$}\label{tableSubgroupsH}
\end{table}

As an example, how the computation was conducted, we provide more details on the case of icosahedral group.

\begin{clai}
(1) The pairwise products of purely imaginary elements of $2I$ generate the whole group $2I$. (2) The pairwise products of elements of $2I$ with real part $\pm\frac{1}{2}$ also generate the whole group $2I$.
\end{clai}

\begin{proof}
There is a Coxeter type description of the group $2I$: this group is generated by three elements $r,s,t$ with the relations
\[
r^2=s^3=t^5=rst=-1,
\]
where one can take $r=i$, $t=(\varphi i+\varphi^{-1}j+k)/2$ so that $\HRe r=\HRe t=0$. The element $-1$ lies in $H_1$ since $-1=r^2$. The relations imply $t=(-1)(t\cdot t)^{-2}$, therefore $t\in H_1$. The generator $r=i$ can be represented as the product $jk$ hence $r\in H_1$. Finally, $s=-t^{-1}s^{-1}\in H_1$, which implies $H_1=2I$. This proves the first statement.

To prove the second statement we note that the pairwise products of Hurwitz units $\frac{1}{2}(\pm1\pm i\pm j\pm k)$ generate the subgroup $2T$ of $2I$. Conjugating $2T$ by all possible elements of $2I$ we get the whole group $2I$ (remember that $H_1$ is \emph{normally} generated by pairwise products of elements with the given real part).
\end{proof}

\begin{table}[h!]
	\centering
\resizebox{15cm}{!} {
\begin{tabular}{c||c|c|c|c|c|c|c|c|c|c|c|}
& $C_2$ & $C_3$ & $C_4$ & $C_6$ & $D_3$ & $D_2$ & $D_4$ & $D_6$ & $T$ & $O$ & $I$ \\
\hline
\hline
$C_2$ & $\Zo_2\times\Zo_2$ & $\Zo_4\times\Zo_6$ & $\Zo_2\times\Zo_4$ & $\Zo_2\times\Zo_6$ & $\Zo_2\times\Zo_2$ & $\Zo_2\times 1$ & $\Zo_2\times 1$ & $\Zo_2\times 1$ & $\Zo_2\times\Zo_3$ & $\Zo_2\times 1$ & $\Zo_2\times 1$ \\
\hline
$C_3$ && $\Zo_2\times\Zo_2$ & $\Zo_6\times\Zo_8$ & $\Zo_2\times\Zo_4$ & $\Zo_2\times\Zo_4$ & $\Zo_6\times2D_2$ & $\Zo_6\times2D_4$ & $\Zo_2\times2D_2$ & $\Zo_2\times 1$ & $\Zo_2\times\Zo_2$ & $\Zo_2\times 1$ \\
\hline
$C_4$ &&& $\Zo_2\times\Zo_2$ & $\Zo_4\times\Zo_6$ & $\Zo_4\times\Zo_2$ & $\Zo_4\times 1$ & $\Zo_2\times 1$ & $\Zo_4\times 1$ & $\Zo_4\times\Zo_3$ & $\Zo_2\times 1$ & $\Zo_4\times 1$ \\
\hline
$C_6$ &&&& $\Zo_2\times \Zo_2$ & $\Zo_4\times \Zo_2$ & $\Zo_6\times 1$ & $\Zo_6\times 1$ & $\Zo_2\times 1$ & $\Zo_4\times 1$ & $\Zo_4\times 1$ & $\Zo_4\times 1$ \\
\hline
$D_3$ &&&&& $\Zo_2\times \Zo_2$ & $\Zo_2\times 1$ & $\Zo_2\times 1$ & $\Zo_2\times 1$ & $\Zo_2\times 1$ & $\Zo_2\times 1$ & $\Zo_2\times 1$ \\
\hline
$D_2$ &&&&&& $1\times 1$ & $1\times 1$ & $1\times 1$ & $1\times \Zo_3$ & $1\times 1$ & $1\times 1$ \\
\hline
$D_4$ &&&&&&& $1\times 1$ & $1\times 1$ & $1\times \Zo_3$ & $1\times 1$ & $1\times 1$ \\
\hline
$D_6$ &&&&&&&& $1\times 1$ & $1\times 1$ & $1\times 1$ & $1\times 1$ \\
\hline
$T$ &&&&&&&&& $1\times 1$ & $1\times 1$ & $1\times 1$ \\
\hline
$O$ &&&&&&&&&& $1\times 1$ & $1\times 1$ \\
\hline
$I$ &&&&&&&&&&& $1\times 1$ \\
\hline
\end{tabular}
}
 \caption{The auxiliary quotients $(2G_1/H_1)\times (2G_2/H_2)$. Groups $G_1$ are listed in first column, $G_2$ are listed in first row.}\label{tableAuxiliary}
\end{table}

Finally, we analyze which elements of $(G_1/H_1)\times (G_2/H_2)$ lie in $H$. Again, in each particular case this is a simple calculation. We demonstrate the idea for one example.

\begin{clai}
$\pi_1(X(C_6,O))$ is trivial.
\end{clai}

\begin{proof}
According to the last column of Table~\ref{tableSubgroupsH}, the quotient $2C_6/H_1$ is isomorphic to the cyclic group $\Zo_4$ with generator $a=\left[\frac{\sqrt{3}}{2}+\frac{1}{2}i\right]$ since $2O$ has elements with real part~$\frac{1}{2}$. By a similar argument, the quotient $2O/H_2$ is trivial.

The group $2C_6/H_1$ is also generated by $a^3=[i]$. However, the element $(i,1)$ lies in $H$. Indeed, this element can be represented as the product $(i,i)\cdot(1,-i)$, and we have $(i,i)\in H$ (since $\HRe i=\HRe i$) and $(1,-i)\in H$ (since $-i\in H_2$). Therefore, $H=2G_1\times 2G_2$, and $\pi_1(X(C_6,O))\cong (2G_1\times 2G_2)/H$ is trivial.
\end{proof}

Similar considerations allow to fill the whole Table \ref{tableFundGrps}, which proves Theorem~\ref{thmFundGrps}. This allows to describe the topology of misorientation spaces.

\begin{thm}\label{thmAllMisSpaces}
The topological types of misorientation spaces are as listed in Table \ref{tableMisSpaces}.
\end{thm}

\begin{proof}
The space $X(G_1,G_2)$ is a closed topological 3-manifold according to Proposition~\ref{propFiniteActionOnMfds}. If $\pi_1(X(G_1,G_2))=1$, then $X(G_1,G_2)\cong S^3$ according to Poincar\'{e} conjecture (Perelman's theorem, see~\cite{Lott}). In general, since the fundamental group is finite, the manifold $X(G_1,G_2)$ is an elliptic manifold, as follows from the elliptization conjecture, also proved by Perelman. The homeomorphism type of an elliptic manifold is uniquely determined by its fundamental group, unless the fundamental group is cyclic~\cite[Thm.4.4.14]{Thurs}. This allows to describe all misorientation spaces with noncyclic fundamental groups.

If, for a closed $3$-manifold $M^3$, there holds $\pi_1(M^3)\cong \Zo_n$, then $M^3$ is homeomorphic to the lens space $L(n;q)$ for some $q$ coprime with $n$. Recall that $L(n;q)$ is the quotient space of $S^3=\{(v,u)\in \Co^2\mid |v|^2+|u|^2=1\}$ by the free action of $\Zo_n$ given by
\[
\epsilon_n(v,u)=(\epsilon_nv,\epsilon_n^pu)
\]
where $\epsilon_n=\cos\frac{2\pi}{n}+i\sin\frac{2\pi}{n}$ is the $n$-th primitive root of unity considered as the generator of $\Zo_n$. The classical result of 3-dimensional topology (see e.g.~\cite{PrYa} and references therein) states that $L(n;q_1)\cong L(n;q_2)$ if and only if $q_1=\pm q_2^{\pm 1}\mod n$. These observations show that, whenever $\pi_1(M^3)=\Zo_n$ for $n=2,3,4,6$, the homeomorphism type is defined uniquely: $M^3\cong L(n,1)$.

By looking at Table~\ref{tableFundGrps}, we see that only three cases are not covered by the preceding arguments. These are the following: (1)~$X(C_2,C_3)$ with fundamental group $\Zo_{12}$; (2)~$X(C_3,C_4)$ with fundamental group $\Zo_{24}$; (3)~$X(C_4,C_6)$ with fundamental group $\Zo_{12}$. We consider these cases separately.
\begin{enumerate}
  \item The group $2C_2$ is the cyclic group generated by $\epsilon_4\in\Co$, the $4$-th primitive root of unity. Similarly, the group $2C_3$ is the cyclic group generated by $\epsilon_6\in\Co$. The action of the generator
\[
a=[\epsilon_4\times \epsilon_6]\in 2C_2\htimes 2C_3\cong \Zo_{12}.
\]
on a quaternion $h=u+jv$ of unit length is given by
\[
ah=\epsilon_4(u+vj)\epsilon_6=\epsilon_4\epsilon_6u + \epsilon_4\epsilon_6^{-1}vj.
\]
Noticing that $\epsilon_4\epsilon_6^{-1}=\epsilon_{12}$ and $\epsilon_4\epsilon_6=\epsilon_{12}^5$, we see that the action of $2C_2\htimes 2C_3$ on $S^3$ coincides with the free action which defines $L(12;5)$.
  \item Similar to the previous case, the action of $[\epsilon_6\times \epsilon_8]\in 2C_3\htimes 2C_4\cong \Zo_{24}$ on $S^3$ is given by $(v,u)\mapsto (\epsilon_6\epsilon_8^{-1}v, \epsilon_6\epsilon_8u)=(\epsilon_{24}z,\epsilon_{24}^7u)$. This coincides with the defining action for $L(24;7)$.
  \item The last exceptional case, $X(C_4,C_6)$ follows from $X(C_2,C_3)$. Indeed, for the groups $G_1=C_4$ and $G_2=C_6$ we have $H_1=C_2$ and $H_2=C_2$. The subgroup $H\subset 2G_1\times 2G_2$ is generated by $(-1,1)$, $(1,-1)$, and $(i,i)$. For the induced action of the group $\hat{H}=H/\{\pm(1,1)\}\cong \Zo_2\times\Zo_2$ on $S^3$ we have $S^3/H\cong S^3$ (for example, this follows from Armstrong theorem (Proposition~\ref{propArmstrong}), also see coordinate explanation in Section~\ref{secCoordinatesSpaces}). The residual action of $(2G_1\times 2G_2)/H$ on $S^3/H\cong S^3$ is free and coincides with the two-sided action of $C_4\htimes C_6\cong \Zo_{12}$ on $S^3$ which defines $X(C_2,C_3)$. The statement follows.
\end{enumerate}

The analysis of these exceptional cases finishes the proof.
%
\end{proof}

\begin{rem}\label{remFreeExplain}
The positions for $X(C_3,D_2)$ and $X(C_3,D_4)$ in Table~\ref{tableMisSpaces} are filled with ``Free''. This indicates that the action of the group $2G_1\htimes 2G_2$ on $S^3$ is free, so that the misorientation spaces are smooth manifolds with fundamental groups $2G_1\htimes 2G_2$ in these cases.

We note that $X(C_3,D_2)$ can be described as follows. The space $F_3$ of full flags in $\Ro^3$ has fundamental group $Q_8=2D_2$ (see Section~\ref{secD2detailed} for more details). Consider the representation of $\Zo_3$ on $\Ro^3$ generated by a 3-rotation over some axis. This representation induces the action of $\Zo_3$ on $F_3$. The latter action is free (whenever the rotation stabilizes the 2-plane of a flag, it acts nontriviall on 1-line, and vice versa). The space $X(C_3,D_2)$ can be naturally identified with the quotient $F_3/\Zo_3$ (this fact is straightforward from the definition of misorientation space).
\end{rem}

The diagonal of Table \ref{tableMisSpaces} implies Proposition~\ref{propSameGroup}.

\begin{rem}\label{remInvolution}
Note that there is a non-free involution $\sigma$ on $S^3\subset \Ho$, sending $h\in S^3$ to $h^{-1}$. It is not difficult to see that $S^3/\sigma\cong D^3$. Similarly, the non-free involution $\sigma\colon\SO(3)\to\SO(3)$, $\sigma(g)=g^{-1}$ yields $\SO(3)/\sigma\cong \Cone\RP^2$. On the other hand, it is impossible to get $\Cone\RP^2$ as the quotient of an involution on $S^3$. Indeed, this involution is necessarily orientation reversing since otherwise the quotient would be a closed topological manifold. But in this case Smith theory implies that the fixed point set is either $S^0$ or $S^2$, and, provided that the embedding of the fixed point set in $S^3$ is tame, the quotient would be either $\Sigma\RP^2$ or $D^3$. Also, it is impossible to get $D^3$ by moding out an involution on $\RP^3$. Indeed, if there exist such involution on $\RP^3$ then its fixed point set contains $S^2=\dd D^3$, but each such sphere divides $\RP^3$ into two non-homeomorphic parts. These observations make Proposition~\ref{propSameGroup} completely consistent with Proposition~\ref{propPatalaSchuh}, the result of \cite{PatSch}, described in the introduction.

\end{rem}

\subsection{Rotations in $\SO(4)$}

\begin{prop}\label{propCriter1con}
Let $G\subset \SO(4)$ be a discrete subgroup and $H$ be the normal subgroup of $G$ generated by stabilizers of the induced action of $G$ on $S^3\subset V=\Ro^4$. The following conditions are equvalent:
\begin{enumerate}
  \item $H=G$;
  \item $\pi_1(S^3/G)=1$;
  \item $G$ is generated by rotations.
\end{enumerate}
\end{prop}

\begin{proof}
Equivalence of (1) and (2) obviously follows from Armstrong theorem (Proposition~\ref{propArmstrong}). To prove that (1) is equivalent to (3), note that the fixed point set $V^A$ of a nontrivial element $A\in \SO(4)$ is either $\{0\}$ or $2$-dimensional. Therefore, $A$ is a rotation if and only if the induced action of $A$ on $S^3$ has fixed points. Therefore, the condition that $G$ is generated by elements with nontrivial stabilizers is equivalent to the condition that $G$ is generated by rotations.
\end{proof}

\begin{cor}\label{corMikhConverse}
The converse to Proposition~\ref{propMikh} holds if $\dim V=4$: if $G\subset\SO(4)$ is not generated by rotations, then $\Ro^4/G\ncong \Ro^4$.
\end{cor}

\begin{proof}
For any point $x\in \Ro^4$, the complement $\Ro^4\setminus\{x\}$ is simply connected. However, if $G$ is not generated by rotations, the space $\Ro^4/G$ is an open cone over $S^3/G$ with apex $y$. Hence we have
\[
\pi_1((\Ro^4/G)\setminus\{y\})\cong\pi_1(S^3/G)\neq 1
\]
according to Proposition~\ref{propCriter1con}.
\end{proof}

\begin{rem}
The complete classification of all discrete subgroups $G\subset\SO(n)$ which satisfy $\Ro^n/G\cong \Ro^n$ is done in the recent work of Lange~\cite{Lange}. Corollary~\ref{corMikhConverse} is covered by this general result.
\end{rem}

Now let $G_1,G_2$ be discrete subgroups of $\SO(3)$, and $G=2G_1\htimes 2G_2$, as before. The group $G$ is considered as discrete subgroup in $\SO(4)\cong \SO(3)\htimes\SO(3)$, as described in Construction~\ref{conSO4}, so we have a natural representation of $G$ on $\Ro^4$. As before, let $H\subseteq G_1\times G_2$ be the normal subgroup, generated by all stabilizers of the two-sided action of $G_1\times G_2$ on a sphere.

\begin{cor}\label{corSpherePseudoref}
$X(G_1,G_2)\cong S^3$ if and only if $2G_1\htimes 2G_2$ is generated by rotations.
\end{cor}

\section{Coordinates on several misorientation spaces}\label{secCoordinatesSpaces}

\subsection{$n$-valued topological groups}\label{subsecNval}

In the following subsections we describe the topological types of several misorientation spaces explicitly. To give the idea of our constructions we recall the observation that appeared in the work~\cite[Sect.4.7]{BuchRees} of Buchstaber and Rees, and review their motivation coming from the theory of $n$-valued groups.

\begin{prop}\label{propBuchRees}
Let $\sigma$ be the conjugation action by the imaginary unit $i$ on $\Sp(1)\cong S^3$. Then the quotient $\Sp(1)/\sigma$ is homeomorphic to $S^3$.
\end{prop}

\begin{proof}
By writing $\Ho$ as $V_1\oplus V_2$, $V_1=\{a+bi\}$, $V_2=\{cj+dk\}$, we see that the conjugation action is trivial on the first summand $V_1$, and is the multiplication by $-1$ on the second summand $V_2$. The unit sphere $\Sp(1)$ is represented as the topological join $S^1_1\ast S^1_2$, where $S^1_i$ is the unit circle in $V_i$ for $i=1,2$. Since $S^1_2/\{\pm1\}\cong S^1$, we have
\[
\Sp(1)/\sigma\cong (S^1_1\ast S^1_2)/\sigma\cong (S^1_1)\ast(S^2/\{\pm1\})\cong S^1\ast S^1\cong S^3,
\]
which finishes the proof.
\end{proof}

Proposition~\ref{propBuchRees} was used in~\cite{BuchRees} to construct the structure of a $2$-valued topological group on the sphere $S^3$. For a general theory of $n$-valued groups we refer the reader to the overview paper~\cite{BuchNval}. In general, if $\Gamma$ is a finite group of automorphisms of a Lie group $W$ (e.g. given by conjugation action), then the quotient $W/\Gamma$ has the natural structure of a $|\Gamma|$-valued topological group, see~\cite[Sect.3(3)]{BuchRees}.

Similarly, according to~\cite[Sect.3(5)]{BuchRees}, given a finite subgroup $\Gamma$ of a Lie group $W$, the bicoset space $\sit{\Gamma}\backslash \jump{W}/\sit{\Gamma}$ admits the natural structure of a $|\Gamma|$-valued topological group. This construction is more relevant to misorientation spaces. It follows directly, that the misorientation space $X(G,G)$ has the structure of the $|G|$-valued group for any discrete subgroup $G\subset\SO(3)$. In particular, from Proposition~\ref{propSameGroup}, it follows that

\begin{prop}
There is a natural structure of $|G|$-valued topological group on $S^3$, determined by $G$ from the list $\{D_2,D_4,D_6,T,O,I\}$. There is a natural structure of $|G|$-valued topological group on $\RP^3$, determined by $G$ from the list $\{C_2,C_3,C_4,C_6,D_3\}$.
\end{prop}

The case when the discrete subgroup $G\subset\SO(3)$ is the cyclic group $C_n$ or the dihedral group $D_n$ for general $n$ is covered by Propositions~\ref{propCyclicCoords}, \ref{propCyclicCoordsDcaseOdd} and \ref{propCyclicCoordsDcaseEven} below. In all these cases there is a $|G|$-valued group structure on the space $X(G,G)$ which is either $S^3$ or $\RP^3$.

\subsection{Coordinates given by consecutive folding}

\begin{prop}\label{propCyclicCoords}
Let $C_n\subset \SO(3)$ be the cyclic group generated by an axial $n$-rotation. Then $X(C_n,C_n)\cong \RP^3$.
\end{prop}

\begin{proof}
As before, let $2C_n\subset S^3$ be the binary extension of $C_n$ and $\varepsilon_{2n}=\cos\frac{2\pi}{2n}+i\sin\frac{2\pi}{2n}\in \Co\subset \Ho$ be the primitive complex root of unity of degree $2n$ considered as the generator of $2C_n$. Encoding a quaternion $h=a+bi+cj+dk$ by a pair $(u,v)\in \Co^2$, where, $u=a+bi$, $v=c+di$, $h=u+vj$, we can describe the two-sided multiplication action as follows:
\begin{equation}\label{eqTwoSidedCyclicAction}
\varepsilon_{2n}^{\alpha}h\varepsilon_{2n}^{\beta}=\varepsilon_{2n}^{\alpha}(u,v)\varepsilon_{2n}^{\beta}=
(\varepsilon_{2n}^{\alpha+\beta}u, \varepsilon_{2n}^{\alpha-\beta}v).
\end{equation}
Let $H_n\subset 2C_n$ be the cyclic subgroup of index $2$ generated by $\varepsilon_{2n}^2$. Then, taking $\alpha=\beta$ or $\alpha=-\beta$ in \eqref{eqTwoSidedCyclicAction}, we see that $H_n\times H_n$ acts on $\Ho=\Co\times\Co$ coordinate-wise. Hence, for the induced action of $H_n\times H_n$ on $S^3=S^1\ast S^1$ we have
\[
(S^1\ast S^1)/(H_n\times H_n)=(S^1/H_n)\ast(S^1/H_n)\cong S^1\ast S^1\cong S^3.
\]
Here $\ast$ is the topological join. The middle homeomorphism constitutes a simple observation that $S^1/H_n\cong S^1$. The residual action of $(2C_n/H_n)\times(2C_n/H_n)\cong \Zt\times\Zt$ on $S^3/(H_n\times H_n)$ has two components: the action of $([\varepsilon_{2n}],[\varepsilon_{2n}])$ on $S^3/(H_n\times H_n)$ is trivial, and the action of $([\varepsilon_{2n}],[1])$ (or $([1],[\varepsilon_{2n}])$) is a free involution on $S^3=S^3/(H_n\times H_n)$, changing sign, according to \eqref{eqTwoSidedCyclicAction}. Hence the total quotient $X(C_n,C_n)=\sit{2C_n}\backslash \jump{S^3}/\sit{2C_n}$ is homeomorphic to $\RP^3$.
\end{proof}

\begin{rem}\label{remCyclCoords}
The proof of Proposition \ref{propCyclicCoords} gives explicit coordinates on the misorientation space $X(C_n,C_n)$. Assume that an orthogonal transformation $A\in \SO(3)$ is given, and the goal is to find its coordinates in the misorientation space $X(C_n,C_n)$. The algorithm runs as follows:
\begin{enumerate}
  \item find a unit quaternion $h=a+bi+cj+dk\in S^3\subset \Ho$, which represents the transformation $A$;
  \item write two complex numbers $u=a+bi$, $v=c+di$, and compute their $n$-th powers $u^n=a'+b'i$, $v^n=c'+d'i$;
  \item write the quaternion $h'=\dfrac{a'+b'i+c'j+d'k}{|a'+b'i+c'j+d'k|}\in S^3$, and consider the corresponding orthogonal transformation $\coord(A)=h'/\{\pm1\}$.
\end{enumerate}
The resulting element $\coord(A)$, written in any convenient form, is the coordinate on $X(C_n,C_n)$ in the sense that $\coord(A)=\coord(\tilde{A})$ if and only if $A=g_1\tilde{A}g_2$ for some $g_1,g_2\in C_n$.
\end{rem}

The argument of Proposition \ref{propCyclicCoords} can be extended to dihedral groups $D_n$ with odd $n$.

\begin{prop}\label{propCyclicCoordsDcaseOdd}
Let $D_n\subset \SO(3)$ be the dihedral rotation group and $n$ be odd. Then $X(D_n,D_n)\cong \RP^3$.
\end{prop}

\begin{proof}
At this time, the binary extension $2D_n\subset S^3$ is generated by the quaternions $\varepsilon_{2n}$ and $j$. Let $H_n$ be the subgroup of $2D_n$ generated by $\varepsilon_{2n}^2$. Then $2D_n/H_n\cong \Zo_4$ (here we used the assumption that $n$ is odd, because for even $n$ the quotient is $\Zt\times\Zt$). As in the proof of Proposition \ref{propCyclicCoords}, the two-sided multiplication action of $H_n\times H_n$ on $S^3$ coincides with the action of $H_n\times H_n$ on $S^1\ast S^1$, which rotates each circle independently, hence $S^3/(H_n\times H_n)\cong S^3$. The residual action of $(2D_n/H_n)\times(2D_n/H_n)\cong C_4\times C_4$ on the sphere $S^3$ coincides with the standard two-sided action, so we may apply Proposition \ref{propCyclicCoords}.
\end{proof}

\begin{rem}\label{remDoddCoords}
The algorithm for finding the coordinates on $X(D_n,D_n)$ for odd $n$ runs as follows. Take a transformation matrix $A\in\SO(3)$, then
\begin{enumerate}
  \item find a unit quaternion $h=a+bi+cj+dk\in S^3\subset \Ho$, which represents the transformation $A$;
  \item write two complex numbers $u=a+bi$, $v=c+di$, and compute their $n$-th powers $u^n=a'+b'i$, $v^n=c'+d'i$;
  \item write another pair of complex numbers $u'=a'+c'i$ and $v'=b'+d'i$, and compute their squares $(u')^2=a''+c''i$, $(v')^2=b''+d''i$;
  \item write the quaternion $h''=\dfrac{a''+b''i+c''j+d''k}{|a''+b''i+c''j+d''k|}\in S^3$, and consider the corresponding orthogonal transformation $\coord(A)=h''/\{\pm1\}$.
\end{enumerate}
\end{rem}

Next, we expand the argument to dihedral groups $D_n$ with even $n$.

\begin{prop}\label{propCyclicCoordsDcaseEven}
Let $D_n\subset \SO(3)$ be the dihedral rotation group and $n$ be even. Then $X(D_n,D_n)\cong S^3$.
\end{prop}

\begin{proof}
As in the previous case, the subgroup $2D_n\subset S^3$ is generated by $\varepsilon_{2n}$ and $j$, and we consider the subgroup $C_n\subset 2D_n$ generated by $\varepsilon_{2n}^2$. Again, we have the homeomorphism $S^3/(H_n\times H_n)\to S^3$, given by $[(u,v)]\mapsto (u'=u^n,v'=v^n)$. The multiplication by $[\varepsilon_{2n}]\in 2D_n/H_n$ from either side of $S^3/(H_n\times H_n)$ acts by $(u',v')\mapsto (-u',v')$ in the new coordinates. Next, we have $(u,v)j=(-v,u)$ and $j(u,v)=(-\bar{v},\bar{u})$. Then, passing to $n$-th powers, and recalling that $n$ is even, we get the left and right actions of the residual element $[j]\in 2D_n/H_n$ in the new coordinates $(u',v')$:
\[
[j](u',v')=(v',u'),\quad (u',v')[j]=(\bar{u}',\bar{v}').
\]
The residual action of $(2D_n\times 2D_n)/(H_n\times H_n)$ on the sphere with coordinates $(u',v')$ has noneffective kernel $\langle(\varepsilon_{2n},\varepsilon_{2n})\rangle$, and the effective part of the action can either switch the coordinates: $(u',v')\leftrightarrow (v',u')$, or conjugate them simultaneously: $(u',v')\leftrightarrow (\bar{u}',\bar{v}')$, or simultaneously change their signs: $(u',v')\leftrightarrow (-u',-v')$. Let us introduce the real coordinates $\tilde{a},\tilde{b},\tilde{c},\tilde{d}$ such that $\tilde{a}+\tilde{b}i=u'+v'$, $\tilde{c}+\tilde{d}i=u'-v'$. In coordinates $(\tilde{a},\tilde{b},\tilde{c},\tilde{d})$ the residual action of $\Zt^3$ is written in a very simple form: it changes signs of even number of coordinates, i.e.
\begin{equation}\label{eqAlternAction}
(\tilde{a},\tilde{b},\tilde{c},\tilde{d})\mapsto (\nu_a \tilde{a},\nu_b \tilde{b},\nu_c \tilde{c},\nu_d \tilde{d}), \quad \nu_a,\nu_b,\nu_c,\nu_d\in\{+1,-1\}, \mbox{ and } \nu_a\nu_b\nu_c\nu_d=+1.
\end{equation}
The quotient space of this action is homeomorphic to a sphere $S^3$ according to Proposition~\ref{propMikhIndex2}.
\end{proof}

\begin{rem}
Note that the last step in the proof of Proposition~\ref{propCyclicCoordsDcaseEven} makes it impossible to write down 4 algebraical coordinates on the space $X(D_{2k},D_{2k})$ as we did in Remarks~\ref{remCyclCoords} and~\ref{remDoddCoords}. Indeed, even in the case of $X(D_2,D_2)$ the ring of invariants has 6 generators subject to some relations, see Subsection~\ref{subsecInvar}. However, the homeomorphism in the last step of the proof can be treated by the following precise formulae. The action of $\Zt^4$ on $S^3$ given by
\[
(\tilde{a},\tilde{b},\tilde{c},\tilde{d})\mapsto (\nu_a \tilde{a},\nu_b \tilde{b},\nu_c \tilde{c},\nu_d \tilde{d}), \quad \nu_a,\nu_b,\nu_c,\nu_d\in\{+1,-1\}
\]
has fundamental domain $P_+=\{a^2+\cdots+d^2=1,a,b,c,d\geqslant 0\}$ which is obviously homeomorphic to $D^3$. The quotient of the smaller group action~\eqref{eqAlternAction} is obtained by gluing two copies of $P_+$ along the boundary, hence the quotient is a sphere (Proposition~\ref{propMikhIndex2} is proved using this idea). Hence, one has to specify the homeomorphism $\varphi\colon P_+\to D^3$ to get the coordinates on $S^3/\Zt^3\cong S^3$.
%
%
\end{rem}

\subsection{Algebraical coordinates given by invariants}

Finally, we outline how the generators of the ring of invariants can be used to construct coordinates on misorientation spaces in some cases.

\begin{con}\label{conCoordInvar}
Assume that the group $G=2G_1\htimes 2G_2\subset\SO(4)$ is generated by rotations (so that $X(G_1,G_2)\cong S^3$ according to Corollary~\ref{corSpherePseudoref}). Assume, moreover, that the algebra of invariants $\Ro[V]^G$, where $V=\Ro^4$, is a free algebra:
\[
\Ro[V]^G\cong \Ro[f_1,f_2,f_3,f_4].
\]
Note that this additional condition does not hold automatically for the groups generated by rotations, see Subsection~\ref{subsecInvar}. Let $n_i=\deg f_i$. The map $V\to V$ given by
\[
(a,b,c,d)\mapsto (f_1(a,b,c,d),\ldots,f_4(a,b,c,d))
\]
is proper and induces the homeomorphism $F\colon V/G\to V'$ to a semialgebraic subset $V'\subseteq V$, see~\cite{ProcShw}. The subset $V'$ contains a neighborhood of the origin (otherwise $0$ would lie on the boundary of $V'$, but the boundary is empty since $V'\cong V$).

Now, if $x=(a,b,c,d)\in S^3$, we can consider the ray $\Ro_+x=\{tx\mid t>0\}$. We have
\[
F(\Ro_+x)=\{(t^{n_1}f_1(a,b,c,d),\ldots,t^{n_4}f_4(a,b,c,d))\mid t>0\}
\]
The set $F(\Ro_+x)$ intersects $S^3$ in a unique point $t_0$. Indeed, the positive solution to the equation
\[
\|F(tx)\|^2=t^{2n_1}f_1^2(a,b,c,d)+\cdots+t^{2n_4}f_4^2(a,b,c,d)=1
\]
is unique since the function on the left is monotone increasing from $0$ to $\infty$ (the coefficients are nonnegative). The solution $t_0$ depends algebraically on $x$. Finally, we have a homeomorphism $S^3/G\to S^3$, given by $x\mapsto F(t_0x)$, which provides coordinates on the space $X(G_1,G_2)=S^3/G\cong S^3$.
\end{con}

\section{Several instances of $X(D_2,D_2)$}\label{secD2detailed}

In this section we recall several ways of understanding that $X(D_2,D_2)$ is homeomorphic to a sphere, which are independent from calculations of Section~\ref{secMisorientationsMain} (based on Poincar\'{e} conjecture) and Proposition~\ref{propCyclicCoordsDcaseEven}. There holds
\[
X(D_2,D_2)=\sit{\Zt^3}\backslash \jump{O(3)}/\sit{\Zt^3}
\]
where $\Zt^3$ is the subgroup of $O(3)$ generated by reflections in three mutually orthogonal planes.

\subsection{Real moment map}
First note that the one-sided quotient $O(3)/\Zt^3$ is the manifold $F_3(\Ro)$ of complete flags in $\Ro^3$ since the latter is the homogeneous space of the Lie group $O(3)$ with the isotropy group $\Zt^3$. On the other hand, if we fix three real numbers $\lambda=(\lambda_1,\lambda_2,\lambda_3)$ with $\lambda_1<\lambda_2<\lambda_3$, then $O(3)/\Zt^3$ is isomorphic to the space
\[
X_{3,\lambda}=\{A\in \Mat_3(\Ro)\mid A^\top=A, A\mbox{ has spectrum }\lambda\}
\]
of symmetric $3\times 3$-matrices with the spectrum $\{\lambda_1,\lambda_2,\lambda_3\}$. The homeomorphism is given by
\[
[Q]\mapsto A=Q\Lambda Q^{\top}=Q\Lambda Q^{-1},
\]
where $Q\in O(3)$, and $[Q]$ is its equivalence class modulo right action of the subgroup $\Zt^3=\{\diag(\pm1,\pm1,\pm1)\}$ of diagonal matrices, and $\Lambda=\diag(\lambda_1,\lambda_2,\lambda_3)$. Let
\begin{equation}\label{eqSymmetricMatrix}
A=A(a_1,a_2,a_3,b_1,b_2,b_3)=\begin{pmatrix}
    a_1 & b_1 & b_3 \\
    b_1 & a_2 & b_2 \\
    b_3 & b_2 & a_3
  \end{pmatrix}
\end{equation}
be a matrix from $X_{3,\lambda}$. The remaining action of $\Zt^3$ on $O(3)/\Zt$ is easily described in terms of matrices: $\Zt^3$ acts on $X_{3,\lambda}$ by conjugation at diagonal orthogonal matrices. In coordinates, we have
\[
(\epsilon_1,\epsilon_2,\epsilon_3)A(a_1,a_2,a_3,b_1,b_2,b_3)=A(a_1,a_2,a_3,\epsilon_1\epsilon_2b_1,\epsilon_2\epsilon_3b_2, \epsilon_3\epsilon_1b_3)
\]
We have $X_{3,\lambda}/\Zt^3\cong \sit{\Zt^3}\backslash \jump{O(3)}/\sit{\Zt^3}=X(D_2,D_2)$.

Consider the map $\mu\colon X_{3,\lambda}\to \Ro^3$ which picks the matrix' diagonal: $\mu(A)=(a_1,a_2,a_3)$. Since the $\Zt^3$ action on $X_{3,\lambda}$ preserves the diagonal, we have an induced map $\widetilde{\mu}\colon X_{3,\lambda}/\Zt^3\to \Ro^3$. The famous Horn--Schur theorem~\cite{Horn} asserts that the image $\mu(X_{3,\lambda})$ coincides with the 2-dimensional permutohedron, i.e. the hexagon
\[
\Pe^2_\lambda=\conv\{(\lambda_{\sigma(1)},\lambda_{\sigma(2)},\lambda_{\sigma(3)})\mid \sigma\in\Sigma_3\},
\]
where $\Sigma_3$ is the permutation group. Alternatively, the hexagon $\Pe^2_\lambda$ can be written by inequalities
\[
\Pe^2_\lambda=\{(a_1,a_2,a_3)\in\Ro^3\mid a_1+a_2+a_3=\lambda_1+\lambda_2+\lambda_3; \lambda_1\leqslant a_i\leqslant \lambda_3\}.
\]
If a diagonal element $a_i$ of a matrix $A$ equals the maximal eigenvalue $\lambda_3$, then all off-diagonal elements of $i$-th column and $i$-th row vanish. If, for example, $a_1=\lambda_3$, then $b_1=b_3=0$ and the block $\begin{pmatrix}a_2 & b_2 \\ b_2 & a_3 \end{pmatrix}$ has eigenvalues $\lambda_1,\lambda_2$. With $a_1=\lambda_3$ and $a_i$'s and $\lambda_i$'s fixed, the value $b_2$ is defined uniquely up to sign. This means that in this case the full preimage $\mu^{-1}((a_1,a_2,a_3))$ consists of two points, forming a $\Zt^3$-orbit of $X_{3,\lambda}$. Hence $\widetilde{\mu}^{-1}((a_1,a_2,a_3))$ is a single point. Similar arguments apply to the case when any other $a_i$ equals either the maximal eigenvalue $\lambda_3$, or the minimal eigenvalue $\lambda_1$. Therefore, for any point $a\in\dd\Pe^2_{\lambda}$ on the boundary of a hexagon, the full preimage $\widetilde{\mu}^{-1}(a)$ consists of a single point.

On the other hand, for the interior point $a=(a_1,a_2,a_3)\in\relint(\Pe_{\lambda}^2)$ of a hexagon, the full preimage $\tilde{\mu}^{-1}(a)$ is homeomorphic to the circle $S^1$. Indeed, the equations on off-diagonal terms of the matrix reduce to an equation of compact real algebraic curve of degree~3, which is homeomorphic to the circle. Hence the whole space $X(D_2,D_2)\cong X_{3,\lambda}/\Zt^3$ is obtained by taking the product $\Pe^2_\lambda\times S^1$, and pinching the circles over the boundary $\dd\Pe^2_\lambda$. This results in the sphere $S^3$.

\begin{rem}
Similar argument was applied by Buchstaber--Terzic in \cite{BT2} in the complex case. They proved the homeomorphism
\begin{equation}\label{eqBuchTerz}
\sit{T^3}\backslash \jump{U(3)}/\sit{T^3}=F_3(\Co)/T^3\cong S^4,
\end{equation}
where $F_3(\Co)$ is the manifold of full complex flags in $\Co^3$. Later, Karhson and Tolman~\cite{KTmain} extended this argument to more general Hamiltonian torus actions of complexity one in general position.
\end{rem}

\subsection{Periodic Toda lattice}

The homeomorphism
\[
\sit{\Zt^3}\backslash \jump{O(3)}/\sit{\Zt^3}\cong S^3
\]
directly follows from the work of van Moerbeke~\cite{VanM}, although this homeomorphism was not stated explicitly in his paper. We review the necessary constructions and provide the missing details in this subsection.

There is a classical dynamical system on the space $X_{3,\lambda}$ of isospectral symmetric matrices. For a matrix $A$ as in \eqref{eqSymmetricMatrix} consider the skew-symmetric matrix
\begin{equation}\label{eqSymmetricMatrix}
P=P(A)=\begin{pmatrix}
    0 & b_1 & -b_3 \\
    -b_1 & 0 & b_2 \\
    b_3 & -b_2 & 0
  \end{pmatrix}
\end{equation}
and define the dynamical system in the form of the Lax pair:
\begin{equation}\label{eqLax}
\dot{A}=[A,P(A)].
\end{equation}
This dynamical system is called \emph{the flow of the periodic Toda lattice}, or shortly, \emph{the Toda flow}. It is not difficult to prove, that dynamical systems written in the form of a Lax pair, preserve the spectrum. Hence, in particular, all solutions to~\eqref{eqLax} lie on the isospectral set $X_{3,\lambda}$ of symmetric matrices. The following facts have straightforward proofs
\begin{itemize}
  \item The flow is $\Zt^3$-invariant.
  \item The quantity $b_1b_2b_3$ is preserved by the Toda flow.
  \item The quantity $b_1b_2b_3$ is preserved by the $\Zt^3$-action.
\end{itemize}
Hence we have the induced flow on the quotient $X_{3,\lambda}/\Zt^3\cong X(D_2,D_2)$ (the flow on the orbifold is the flow, which is smooth on each smooth stratum of the orbifold). Moreover, we can consider the map $p\colon X_{3,\lambda}\to \Ro$, $p([A])=b_1b_2b_3$, as well as the induced map $\widetilde{p}\colon X_{3,\lambda}/\Zt^3\to \Ro$. Van Moerbeke \cite{VanM} proved that the image $p(X_{3,\lambda})$ is a closed interval $[-m_1,m_2]$, $m_1,m_2>0$, and, for a point $r\in [-m_1,m_2]$, the full preimage is homeomorphic to
\begin{equation}\label{eqPreimagesOfProduct}
\widetilde{p}^{-1}(r)\cong\begin{cases}
            \mbox{a torus }T^2, & \mbox{if } r\neq0,-m_1,m_2 \\
            \mbox{a circle }T^2/T^1_1, & \mbox{if } r=-m_1 \\
            \mbox{a circle }T^2/T^1_2, & \mbox{if } r=m_2 \\
            \mbox{a torus} T^2 \mbox{ made of 3 hexagons}, & \mbox{if } r=0,
\end{cases}
\end{equation}
where $T^2=T^1_1\times T^1_2$ is a coordinate splitting. The Toda flow is known to be a Liouville integrable system (see, e.g. \cite{VanM,KrichOld}). All tori appearing in~\eqref{eqPreimagesOfProduct}, except for $r=0$ are the Liouville--Arnold tori of the Toda flow. The torus $\widetilde{p}^{-1}(r)$ is exceptional: the periodic Toda flow on the set $\{A\in X_{3,\lambda}\mid b_1b_2b_3=0\}$ degenerates to an open Toda flow, and the trajectories at this set exhibit nice asymptotic and combinatorial properties, studied in~\cite{Tomei}. The quotient $\{A\in X_{3,\lambda}\mid b_1b_2b_3=0\}/\Zt^3$ of the exceptional set is subdivided into 3 natural strata: $\{b_1=0\}$, $\{b_2=0\}$, and $\{b_3=0\}$. Each of these strata is a hexagon and they are glued together as shown on Fig.~\ref{figHexes}.

\begin{figure}[h]
\begin{center}
\includegraphics[scale=0.2]{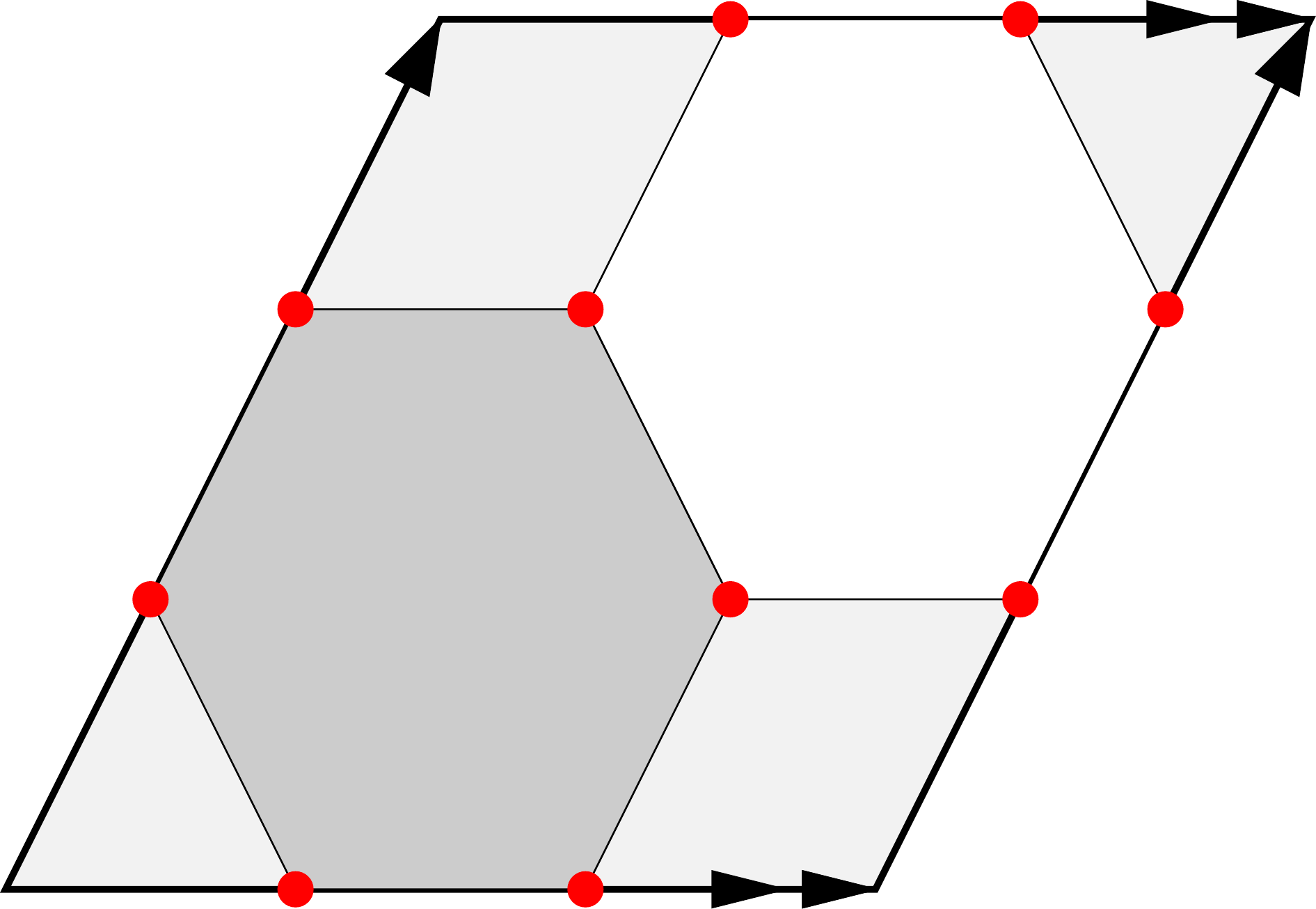}
\end{center}
\caption{The subdivision of $\widetilde{p}^{-1}(0)/\Zt^3\cong T^2$ into strata $\{b_1=0\}$, $\{b_2=0\}$, and $\{b_3=0\}$}\label{figHexes}
\end{figure}

The detailed study of both open and periodic Toda flow from topological point of view was done by the first author in~\cite{AyzMatr}. In that paper we concentrated on the complex case of general periodic tridiagonal matrices, in particular, we reproved the homeomorphism \eqref{eqBuchTerz}.

The real version is completely similar. Here we are interested in the case of real $(3\times 3)$-matrices. The description \eqref{eqPreimagesOfProduct} implies that $\widetilde{p}^{-1}([0,m_2])$ is a solid torus $D^2\times S^1$, and $\widetilde{p}^{-1}([-m_1,0])$ is a solid torus $S^1\times D^2$. Hence $X_{3,\lambda}/\Zt^3\cong X(D_2,D_2)$ is obtained from patching two solid tori along the boundary. However, since distinct circles $T^1_1$ and $T^1_2$ are collapsed over $-m_1$ and $m_2$, the parallel of the first solid torus is patched to the meridian of the second, and vice versa. Hence the result is a $3$-sphere patched of two solid tori (the one corresponding to the nonnegative products $b_1b_2b_3$ and another corresponds to nonpositive products $b_1b_2b_3$).

\subsection{The argument with the ring of invariants}\label{subsecInvar}

The space $X(D_2,D_2)$ is the quotient of two-sided action of $Q_8=\{\pm1,\pm i,\pm j,\pm k\}$ on the sphere $S^3$ of unit quaternions. This particular action was studied in details in~\cite[Thm.3.1]{Mikh}. The action of $G=Q_8\times Q_8$ on $\Ho=\{a+bi+cj+dk\}\cong V= \Ro^4$ is given by Klein four-group action on $(a,b,c,d)$ and sign-changes of even number of coordinates. This representation is generated by rotations. However, the straightforward analogue of Chevalley--Shephard--Todd theorem does not hold. The algebra of invariants $\Ro[V]^G$ is not free: it is generated by the polynomials
\begin{align}\label{eqInvariants}
  f_0&=a^2+b^2+c^2+d^2 & f_1&=a^2b^2+c^2d^2 \nonumber\\
  f_2&=a^2c^2+b^2d^2 & f_3&=a^2d^2+b^2c^2 \\
  f_4&=abcd & f_5&=a^2b^2c^2+a^2b^2d^2+a^2c^2d^2+b^2c^2d^2.\nonumber
\end{align}
subject to relations
\begin{equation}\label{eqRelations}
f_1f_2+f_2f_3+f_1f_3+4f_4^2=f_0f_5;\qquad f_1f_2f_3=f_4^2(f_0^2-4(f_1+f_2+f_3))+f_5^2.
\end{equation}
However, Mikhailova used~\eqref{eqInvariants} and~\eqref{eqRelations} to construct the coordinates on $V/G$, and, in particular, to prove the homeomorphism $V/G\cong V$ stated in her general theorem. The coordinates can be restricted to $S^3\subset V=\Ho$.

This observation allows to prove Proposition~\ref{propMikh} for $\dim V=4$, and general discrete group generated by rotations, which explains the importance of the case $G_1=G_2=D_2$.

\section*{Acknowledgements}

The authors thank Victor Buchstaber for his valuable comments, especially for his observation about the structure of multi-valued topological groups on misorientation spaces $X(G,G)$. We highly appreciate the discussion with Ernest Vinberg who helped us find the relevant sources concerning invariant theory of finite group actions.

\end{document}